\documentclass[11pt,reqno,a4paper]{amsart}
\usepackage[T1]{fontenc}
\usepackage[utf8]{inputenc}
\usepackage{fbb}

\usepackage{amsmath,amsthm,amssymb,mathtools}
\allowdisplaybreaks
\usepackage[top=2.5cm,bottom=2.4cm,left=2.6cm,right=2.6cm,headsep=0.2in]{geometry}
\usepackage{fancyhdr}
\usepackage{microtype}
\usepackage{enumitem}
\usepackage[noadjust]{cite}
\usepackage{xcolor}
\usepackage{hyperref}
\hypersetup{
  colorlinks=true,
  linkcolor=blue!50!green,
  citecolor=blue!50!green,
  urlcolor=blue!50!green,
  pdftitle={Spectral Dehn Functions: Hyperbolicity, Filling Geometry and Diffusion},
  pdfauthor={Mayukh Mukherjee}
}
\numberwithin{equation}{section}
\theoremstyle{plain}
\newtheorem{theorem}{Theorem}[section]
\newtheorem{lemma}[theorem]{Lemma}
\newtheorem{proposition}[theorem]{Proposition}

\newtheorem{introthm}{Theorem}

\theoremstyle{definition}

\theoremstyle{remark}

\setlist[enumerate]{label=\textup{(\roman*)},leftmargin=2em}
\title[Dirichlet spectra of van Kampen diagrams]{Spectral Dehn Functions:\\ Hyperbolicity, Filling Geometry and Diffusion}
\author{Mayukh Mukherjee}
\address{Department of Mathematics, Indian Institute of Technology Bombay, Mumbai 400076, India}
\email{mukherjee@math.iitb.ac.in, mathmukherjee@gmail.com}
\date{November 2026}

\newcommand{\Area}{\operatorname{Area}}

\newcommand{\E}{\mathbb E}
\newcommand{\Z}{\mathbb Z}
\newcommand{\R}{\mathbb R}
\begin{document}
\begin{abstract}
We show that Dirichlet eigenvalues of van Kampen diagrams can distinguish groups whose Dehn functions have the same order. The Dehn function bounds the least filling area in terms of word length. We use the random walk on the faces of a filling, stopped at the exterior. For each null word we take the infimum, over its fillings, of the reciprocal of the first Dirichlet eigenvalue, then maximize over words of length at most $n$. We call this the spectral Dehn function. Its coarse growth is a quasi-isometry invariant for presentations whose generators form a word ball and whose relators are all null words of length at most three. In such a presentation of $\Z^2$, the boundary word of a square can have a filling with more faces and a larger first eigenvalue than any of its least-area fillings. For two groups with quartic Dehn functions, the spectral Dehn function is at most $Cn^4$ for one and at least $cn^6/(\log n)^2$ for the other.

The spectral Dehn function is bounded exactly for word-hyperbolic groups. For other finitely presented groups it grows at least as fast as $n^2/(\log n)^2$. In these presentations the bound is sharp for non-hyperbolic automatic groups and for non-hyperbolic groups acting properly and cocompactly by isometries on proper geodesic Lipschitz $1$-connected spaces. For the face walk we maximize the expected absorption time over starting faces. In these classes its supremum over least-area fillings of cyclically reduced null words of length at most $n$ has order $n^2$. Taking the infimum over all fillings of each null word before maximizing over words of length at most $n$ gives order $n^2/\log n$ in these presentations. Under a quadratic isoperimetric inequality, every non-hyperbolic group has, for each sufficiently large $n$, a null word of length at most $n$ whose least-area fillings have, for the face and vertex walks, $j$th Dirichlet eigenvalue of order $j/n^2$ for $1\le j\le cn^2$.
\end{abstract}
\maketitle

\section{Introduction}\label{sec:intro}

Let $\mathcal P=\langle S\mid R\rangle$ be a finite presentation of a group $G$. A word $w$ representing the identity in $G$, a \emph{null word}, can be filled by a van Kampen diagram $\Delta$ over $\mathcal P$. We write $\partial\Delta=w$ when its outer boundary reads $w$, and denote the number of its faces by $\Area(\Delta)$. The Dehn function $\delta_{\mathcal P}(n)$ is obtained by taking the least area of a filling of each null word, and then the maximum over words of length at most $n$.

We use the following walk on the faces of a filling. At each step we choose a side of the current face uniformly at random and cross it. The walk stops if it crosses into the exterior. Write $\mu_1(\Delta)$ for the first Dirichlet eigenvalue of this walk, that is, the smallest eigenvalue of $1-P$, where $P$ is the transition matrix restricted to the faces. We define the \emph{spectral Dehn function} by
\begin{equation}\label{eq:spectral-dehn-definition}
 \Theta_{\mathcal P}(n)=
 \max_{\substack{|w|\le n\\ w=_G1}}\;
 \inf_{\partial\Delta=w}\frac{1}{\mu_1(\Delta)} .
\end{equation}
A filling with no faces contributes zero. The infimum is over all fillings of the given word, without any bound on their area.

We write $f\lesssim g$ if $f\le Cg$ for some constant $C$, and $f\asymp g$ if $f\lesssim g$ and $g\lesssim f$. The constants may depend on the presentations and the stated geometric hypotheses, but not on the variables in the estimates.

To compare different groups we use the following presentations. Given a finite generating set $S$ of $G$, let $\mathcal P_r(G,S)$ have one formal generator for each element of the closed word ball of radius $r$ about the identity, including a generator for the identity itself. Its relators are all words of length one, two or three in these generators and their inverses that represent the identity in $G$. We call $\mathcal P_r(G,S)$ an \emph{admissible Rips presentation} of $G$ if the natural map to $G$ is an isomorphism.

Every sufficiently large radius is admissible. To see this, we fix a finite presentation of $G$ on $S$ and take $r$ at least one and at least the length of every defining relator. The prefixes of a geodesic from the identity to a ball element stay in the ball, as do all the prefixes of each defining relator. The triangular relations along these prefixes, together with the relator for the identity generator, express the ball generators in the original letters and imply all the defining relators. For each fixed choice of an admissible Rips presentation $\mathcal P_G$, we write $\Theta_G=\Theta_{\mathcal P_G}$.

Consider the word $w_m=a^m b^m a^{-m}b^{-m}$ around a square in $\Z^2$. In an admissible Rips presentation containing the standard generators $a,b$, every least-area filling of $w_m$ has $\mu_1^{-1}\asymp m^2$, while the infimum of $\mu_1^{-1}$ over all fillings of $w_m$ has order $m^2/(\log m)^2$.

\begin{introthm}[Quasi-isometry invariance and a gap]\label{thm:intro-invariant}
Let $G,H$ be quasi-isometric finitely presented groups, each with a fixed admissible Rips presentation. There is a constant $C\ge1$ such that for all $n\ge1$,
\[
 \Theta_G(n)\le C\bigl(1+\Theta_H(\lceil Cn\rceil)\bigr),\qquad
 \Theta_H(n)\le C\bigl(1+\Theta_G(\lceil Cn\rceil)\bigr).
\]
For any finite presentation $\mathcal P$ of $G$, the function $\Theta_{\mathcal P}$ is bounded if and only if $G$ is word-hyperbolic. If $G$ is not word-hyperbolic, then for every sufficiently large $n$,
\[
                  \Theta_{\mathcal P}(n)\ge c\,n^2/(\log n)^2 .
\]
\end{introthm}

In particular, the coarse growth of $1+\Theta_G$ does not depend on the admissible Rips presentation of $G$. Proposition~\ref{prop:first-face-qi} also gives
\[
 \Theta_G(n)\le C\bigl(1+\Theta_{\mathcal P}(\lceil Cn\rceil)\bigr)
\]
for any finite presentation $\mathcal P$ of $G$ with a nonempty generating set. Thus $\Theta_G$ is either bounded or grows at least as fast as $n^2/(\log n)^2$. This is the spectral counterpart of the gap between linear and quadratic Dehn functions~\cite[Theorem~2.3.F, pp.~103--104]{GromovHyperbolic1987}. Theorem~\ref{thm:intro-optimized} shows that the order $n^2/(\log n)^2$ is attained.

For the lower bound, we need a word for which \emph{every} filling has a small first eigenvalue. In a non-hyperbolic Cayley graph we find, at each sufficiently large scale $n$, a simple loop of length at most $n$ with two pairs of opposite arcs far apart. Two truncated distance functions give a map to a square of side comparable to $n$. The image of the loop goes once around the boundary of the square. In any filling we use this map on the edges and extend it over the faces. The resulting image contains the whole square, so the filling has many faces far from its exterior. We use these faces to construct test functions.

The ordinary Dehn function gives upper and lower bounds for the spectral one. Let $\Delta$ be a filling of a word of length at most $n$. If $\Delta$ has a face, testing with the function equal to one on the faces gives $\mu_1^{-1}\ge\Area(\Delta)/n$. When there are no faces the same inequality holds by convention. Hence $\Theta_{\mathcal P}(n)\ge\delta_{\mathcal P}(n)/n$. In the other direction, if $\delta_{\mathcal P}(n)\lesssim n^\alpha$ for some $\alpha>1$, we apply Lemma~\ref{lem:quadratic-nash} to a least-area filling. The total number of face sides, counted with multiplicity, is at most a constant times its area, so $\Theta_{\mathcal P}(n)\lesssim n^{2\alpha-2}$. For a quartic Dehn function we thus have $n^3\lesssim\Theta_{\mathcal P}(n)\lesssim n^6$. The next theorem shows that the spectral Dehn functions of two groups with quartic Dehn functions can have different growth.

\begin{introthm}[Groups with quartic Dehn function]\label{thm:intro-examples}
There are finitely presented groups, Bridson's double $B_3$ and a group $Q$, both with Dehn function of order $n^4$, such that for any fixed admissible Rips presentations of these groups and all sufficiently large $n$,
\[
          \Theta_{B_3}(n)\le C n^4,\qquad
          \Theta_Q(n)\ge c\,n^6/(\log n)^2 .
\]
\end{introthm}

For the double we work first in its usual presentation~\cite{BridsonPolynomial2002,BridsonRiley2009}. We construct paths in the dual graph from the faces to the exterior and bound the number that cross each edge. For the second group we use an example of Sapir, Birget and Rips~\cite[Corollary~1.4, p.~352]{SapirBirgetRips2002}. It has words of length at most $n$ for which every filling has intrinsic diameter at least $cn^3$. This gives a face far from the exterior in the dual graph, from which we construct a test function. For fillings of very large area the constant function on the faces suffices instead. Combining Theorem~\ref{thm:intro-examples} with the bounds from the Dehn function gives $cn^3\le\Theta_{B_3}(n)\le Cn^4$ and $cn^6/(\log n)^2\le\Theta_Q(n)\le Cn^6$.

We next ask when the lower bound in Theorem~\ref{thm:intro-invariant} is sharp. We also consider higher eigenvalues and fillings with a bound on the area. We list the Dirichlet eigenvalues on faces as $\mu_1\le\mu_2\le\cdots$, with multiplicity. For $F\in[0,\infty]$, we define $\Theta_{\mathcal P,F}^{(j)}(n)$ by replacing $1/\mu_1$ with $1/\mu_j$ in \eqref{eq:spectral-dehn-definition} and taking the infimum only over fillings of area at most $F$. A filling with fewer than $j$ faces contributes zero. If no filling has area at most $F$, the infimum is $+\infty$. Thus $\Theta_{\mathcal P,\infty}^{(1)}=\Theta_{\mathcal P}$.

For the face walk, let $\tau$ be the first time the walk reaches the exterior. We write $T(\Delta)=\max_x\mathbb E_x\tau$, where the maximum is over the starting faces, and set $T(\Delta)=0$ when $\Delta$ has no faces.

A geodesic metric space is \emph{Lipschitz $1$-connected} if there is a constant $K$ such that each $L_0$-Lipschitz map of the unit circle extends to a $KL_0$-Lipschitz map of the unit disk. The extension property for maps of $S^0$ follows by joining the two image points by a geodesic.

\begin{introthm}[Sharp bounds for automatic groups and Lipschitz $1$-connected spaces]\label{thm:intro-optimized}
Let $G$ be non-hyperbolic. Suppose that $G$ is automatic, or acts properly and cocompactly by isometries on a proper geodesic Lipschitz $1$-connected space. Fix an admissible Rips presentation $\mathcal P$ of $G$. Then for large $n$,
\[
                   \Theta_G(n)\asymp n^2/(\log n)^2 .
\]
More generally, there are constants $A,c>0$ such that for all sufficiently large $n$, $An^2\le F\le\infty$ and $1\le j\le cn^2$,
\[
 \Theta_{\mathcal P,F}^{(j)}(n)\asymp\frac{n^2}{j+D(n,F)^2},\qquad
 D(n,F)=1+\min\{\log(n+2),\log(1+F/n^2)\}.
\]
For the face walk, $\max_{|w|\le n,\,w=_G1}\inf_{\partial\Delta=w}T(\Delta)\asymp n^2/\log n$.
For each $\varepsilon>0$, every null word of length at most $n$ has a filling with $O_\varepsilon(n^{2+\varepsilon}/(\log n)^2)$ faces on which both $\mu_1^{-1}=O_\varepsilon(n^2/(\log n)^2)$ and $T=O_\varepsilon(n^2/\log n)$.
\end{introthm}

The geometric hypothesis holds for CAT(0) groups and for the integer Heisenberg groups $H_{2q+1}(\Z)$ with $q\ge2$. The dual graph of the disk we construct has concentric cycles. Most of the cycles are arranged in blocks. Their lengths are constant within each block and double from one block to the next. After the last block, we add shorter cycles if needed to give the prescribed perimeter. A first eigenfunction is constant on each cycle, so we can estimate the eigenvalue on a weighted path. We use the Lipschitz extension and the cocompact action, or the combing in an automatic group, to label the vertices of the disk by group elements. Adjacent labels differ by words of bounded length. With these words on the edges, we fill the resulting face words in the fixed presentation. By Theorem~\ref{thm:intro-invariant}, the conclusion for $\Theta_G$ also holds for finitely presented groups quasi-isometric to these groups.

We now restrict to diagrams of least area for their boundary word, which we call \emph{minimal}. On such a diagram we also consider the vertex walk. At each step it chooses an incident edge uniformly at random, counting each loop twice, and stops when it reaches a boundary vertex. For the vertex walk, $T(\Delta)$ is the largest expected time to reach the boundary over starting vertices in the interior. We set $T(\Delta)=0$ when there are no interior vertices.

We write $\lambda_j$ for the Dirichlet eigenvalues of the vertex walk, in increasing order with multiplicity. For a finite presentation $\mathcal P$, define the \emph{least-area vertex and face profiles} $\Lambda_{\mathcal P}^{(j)}(n)$ and $\Lambda_{\mathcal P}^{*,(j)}(n)$ as the infima of $\lambda_j$ and $\mu_j$, respectively, over minimal diagrams of cyclically reduced null words of length at most $n$. In the vertex infimum we consider only diagrams with at least $j$ interior vertices, and in the face infimum only diagrams with at least $j$ faces. We set $\inf\varnothing=+\infty$ and omit $(j)$ when $j=1$.

\begin{introthm}[Least-area fillings under a quadratic isoperimetric inequality]\label{thm:intro-quadratic}
Let $\mathcal P$ be a finite presentation of a non-hyperbolic group $G$. Suppose that $\delta_{\mathcal P}(n)\le Kn^2$. For every sufficiently large $n$ there is a cyclically reduced null word $w_n$ of length at most $n$ such that on every minimal filling $\Delta$ of $w_n$,
\[
 \lambda_j(\Delta)\asymp\mu_j(\Delta)\asymp j/n^2
                      \qquad (1\le j\le cn^2).
\]
On every such filling $T(\Delta)\asymp n^2$ for either walk. In the same range of indices both $\Lambda_{\mathcal P}^{(j)}(n)$ and $\Lambda_{\mathcal P}^{*,(j)}(n)$ have order $j/n^2$. For either walk the supremum of $T$ over all minimal fillings of cyclically reduced null words of length at most $n$ has order $n^2$.
\end{introthm}

The upper bounds for the eigenvalues follow by pulling back functions supported in disjoint subsquares. For the lower bounds, every disk subdiagram of a least-area diagram is itself least-area. The quadratic isoperimetric inequality for the presentation thus gives an isoperimetric inequality for sets of interior vertices and for sets of faces, from which we obtain a Nash inequality for each walk. Without a quadratic isoperimetric inequality we still obtain spectral estimates on every filling of a suitable word, as well as a hyperbolicity criterion using only least-area diagrams (Theorem~\ref{thm:intro-A}).

Area and intrinsic diameter are classical measures of a filling~\cite[Section~5]{Gromov1996GeometricGroupTheory}, \cite[Part~II, pp.~86--87]{Riley2003}. Gersten and Riley also studied the diameter of the dual graph~\cite{GerstenRileyDuality2005}. The domains here are diagrams filling a prescribed word, rather than subsets of the Cayley graph. The exact and monotone isoperimetric profiles compare boundaries of finite vertex sets~\cite{MukherjeeExactProfiles}, and the Poincar\'e and Dirichlet--Poincar\'e profiles use finite subgraphs~\cite{HumeMackayTessera2022,Hume2019}. We leave the corresponding estimates for the $p$-Laplacian to a separate paper.

\subsection{Organization of the paper}

In Section~\ref{sec:square} we map a loop in the Cayley graph to the boundary of a square and use the map to find faces and vertices far from the boundary of any filling. In Section~\ref{sec:minimum} we prove the least-area estimates of Theorem~\ref{thm:intro-quadratic}. Section~\ref{sec:universal} gives the estimates on arbitrary fillings used in Theorem~\ref{thm:intro-invariant}. In Section~\ref{sec:optimized} we prove Theorem~\ref{thm:intro-optimized} and complete the proof of Theorem~\ref{thm:intro-invariant}, including quasi-isometry invariance. The two quartic examples in Theorem~\ref{thm:intro-examples} are treated in Section~\ref{sec:separation}.

\section{A square covered by every filling}\label{sec:square}

We first find a loop in the Cayley graph that maps to the boundary of a square. We use this map to find many vertices and faces far from the boundary of every filling of the loop.

In a van Kampen diagram the faces are the bounded regions cut out by the edges. The oriented edges carry generator labels, and the boundary of each face reads a cyclic conjugate of a relator or its inverse~\cite[Part~II, Chapter~1, Section~1.1, p.~90]{Riley2003}. A boundary path can visit a vertex or an edge more than once. In particular, a bridge on the outer boundary occurs twice in the boundary path~\cite[Chapter~V, Section~1, p.~236]{LyndonSchupp}. We count these occurrences separately. We omit empty relators but allow monogons, digons and relators that are not cyclically reduced. We fix $L\ge1$ to bound the lengths of the defining relators, and hence the lengths of all face boundaries.

The vertex walk runs on the $1$-skeleton of the diagram. The dual graph has one vertex for each face and one for the exterior. Each edge of the diagram gives a dual edge joining the vertices for the regions on its two sides, with a loop if they coincide. The face walk runs on the dual graph. With no relators all diagrams are trees. The least-area profiles $\Lambda_{\mathcal P}$ and $\Lambda^*_{\mathcal P}$ are then $+\infty$, whereas a filling with no faces contributes zero to the spectral Dehn function.

\begin{proposition}[Square coordinates]\label{lem:square-coordinates}\label{prop:all-scale-square}
Fix a finite presentation of a non-hyperbolic group $G$, and let $X$ be its Cayley graph. For every sufficiently large $n$, $X$ contains a simple cycle of length at most $n$, labelled by a cyclically reduced null word $w_n$. There is a $1$-Lipschitz map $\Phi:X\to Q=[0,r]^2$, where $r\asymp n$ and $Q$ has the maximum metric. The cycle splits into four successive arcs, and $\Phi$ takes each arc into the corresponding side of $Q$, with endpoints at the corners in order. For any filling of $w_n$, the edge labels determine a map of its $1$-skeleton to $X$, sending its basepoint to the start of the cycle. Under the composition with $\Phi$, every point of $Q$ lies within $L$ of a vertex image. At least $cn^2$ vertices are at distance at least $cn$ from the outer boundary, and at least $cn^2$ faces are at distance at least $cn$ from the exterior vertex in the dual graph.
\end{proposition}

\begin{proof}
We look for a short loop with two pairs of opposite arcs far apart in the Cayley graph. We obtain it from a geodesic triangle with a point far from the other two sides. Papasoglu showed that in a non-hyperbolic group there are, for arbitrarily large $n$, loops of length at most $n$ such that every filling has a vertex whose image in the Cayley graph lies at distance at least $cn$ from the loop~\cite[Proposition~1, p.~1269]{Papasoglu1998}. Here we need a triangle of perimeter at most $4R$ that is not $\eta$-thin, with $\eta$ comparable to $R$, for \emph{every} large $R$.

The local-to-global theorem of Kapovich and Kleiner~\cite[Appendix, Lemma~8.2 and Theorem~8.3]{OlshanskiiOsinSapir2009} gives universal constants $C_i$: if the simplicial Rips complex $P_d(G)$ is simply connected, $d\ge C_1\sigma$, and every ball of radius $C_2d$ is $\sigma$-hyperbolic in the Gromov product sense at some basepoint of the ball, then $G$ is $C_3d$-hyperbolic. Here the balls have the ambient word metric. For $d\ge L$, we can replace each Rips edge by a Cayley graph geodesic with the same endpoints. All vertices on the geodesic span a single simplex, so the replacement is a homotopy with fixed endpoints. The resulting Cayley graph loop has a van Kampen filling, which maps into $P_d(G)$ since the vertices of each defining relator loop span a simplex. Thus $P_d(G)$ is simply connected. For each large $R$ we set $d=R/C_2$ and $\sigma=d/C_1$. Since $G$ is non-hyperbolic, the product inequality fails at every basepoint in some ball of radius $R$. In particular, at its centre $p$ there are points $x,y,z$ in the ball with
\[
 (x,y)_p<q-\sigma,\qquad q=\min\{(x,z)_p,(y,z)_p\},\qquad
 (x,y)_p=\tfrac12(d(p,x)+d(p,y)-d(x,y)).
\]
We take geodesics in the Cayley graph with its path metric, using the same side $[p,z]$ in the triangles $pxz$ and $pyz$. Set $\eta=\sigma/5$ and $t=q-2\eta$. The displayed inequality gives $q>\sigma$, and $t$ is no greater than any of the three side lengths from $p$. Every point of $[x,z]$ or $[y,z]$ is at distance at least $q$ from $p$. If both triangles were $\eta$-thin, the points on $[p,x]$ and $[p,y]$ at distance $t$ from $p$ would therefore each be within $\eta$ of $[p,z]$, since $q-t=2\eta$. They would then each be within $2\eta$ of the point at distance $t$ on $[p,z]$. This would give $(x,y)_p\ge t-2\eta>q-\sigma$. Both triangles have perimeter at most $4R$, and one has a point on a side at distance at least $\eta$ from the other two sides.

We choose a vertex $v$ within one step of this point on the same side. Starting at $v$, we follow the side in each direction until we meet one of the other two sides. A simple path in their union joins the two points just reached. Together with the segment through $v$ it gives a simple cycle no longer than the perimeter of the triangle. The return path stays at distance at least $\eta-1$ from $v$.

Let $r=\lfloor(\eta-1)/4\rfloor$. On either side of $v$, we mark the vertices $r$ and $2r$ steps away. These four vertices divide the cycle into four arcs. We call the arc through $v$ $A_1$ and number the other three in order around the cycle. The return path stays at distance at least $3r$ from $A_1$. The points of $A_3$ on the geodesic are at least $2r$ from $v$, so they are at least $r$ from $A_1$. Since $A_2$ and $A_4$ lie on the geodesic on opposite sides of $v$, we also have $d(A_2,A_4)\ge2r$. For a vertex $y$ of $X$ we set
\[
 \Phi(y)=\big(\min\{r,d(y,A_4)\},\ \min\{r,d(y,A_1)\}\big),
\]
and interpolate linearly along edges. Both coordinates are $1$-Lipschitz. The arcs $A_1,A_2,A_3,A_4$ map to the bottom, right, top and left sides of $Q$, respectively, from corner to corner. This is the distance map in the Besicovitch argument~\cite[Proposition~9.103, p.~331]{DrutuKapovich2018}; Hume and Mackay use the same two truncated distances in their work on separation~\cite[p.~4662, before Lemma~5.4]{HumeMackay2020}. Taking $R=n/4$ gives the required length. The word read along the cycle is cyclically reduced since the cycle is simple.

Now let $\Delta$ be any filling of this word, with its $1$-skeleton mapped to $X$ as in the statement. The outer boundary maps once around the chosen simple cycle. If the boundary path repeated a vertex before closing, its image would repeat a vertex too. Thus the diagram is a topological disk~\cite[Definition~1.5, p.~1106]{Bridson1999}. We compose the map with $\Phi$ and extend it across the faces. In the characteristic polygon for each face, we send the centre to a point in the convex hull of the boundary image and extend linearly along the radii. The polygon has a separate side for each occurrence of a boundary edge. The boundary map may turn back along a side of $Q$, but it goes from corner to corner in order and so winds once about every interior point. If the extension missed such a point, the boundary map would contract in the punctured square. Thus the extension covers $Q$.

For each face we choose the coordinates of one of its corners. The image of the face has diameter at most $L$. To any point of $Q$ we assign a closed face containing one of its preimages. The chosen corner then gives a vertex image and a face coordinate within $L$ of the point. To count vertices and faces in a subsquare of side $s\ge8L$, we place a grid of points $3L$ apart and leave a margin of $2L$ on all sides. Distinct grid points cannot give the same vertex or the same face, since each assigned point is within $L$ of its grid point. The margin keeps all the assigned coordinates inside the subsquare. This gives at least $cs^2$ vertices and at least $cs^2$ faces.

We apply the count to $[r/4,3r/4]^2$. A vertex with coordinates in this subsquare lies at distance at least $cn$ from the outer boundary, since its coordinates change by at most one along each edge. In the dual graph, two faces sharing a side have chosen points at distance at most $2L$. Before reaching the exterior a path must reach a face with a side on the outer boundary, whose chosen point is within $L$ of $\partial Q$. Thus a face with coordinates in the middle subsquare is at distance at least $cn$ from the exterior.
\end{proof}

\section{Least-area diagrams}\label{sec:minimum}

For the upper bounds in Theorem~\ref{thm:intro-quadratic}, we pull back test functions from the square. The lower bounds come from the isoperimetric inequality on subdiagrams of a least-area filling.

The absorbing set consists of the vertices where the walk stops: the boundary vertices for the vertex walk and the exterior vertex of the dual graph for the face walk. We write $I$ for the other vertices of the graph. Let $m(x)$ be the degree of $x$ in the full graph, including edges to the absorbing set and counting loops twice. Thus $m(f)=|\partial f|$ on the dual graph. Put $M=\sum_{x\in I}m(x)$. Each vertex in $I$ has a path to the absorbing set. For a function $u$ vanishing on the absorbing set, we set
\[
 \mathcal E(u)=\sum_{xy\,{\rm edge}}(u(x)-u(y))^2,\qquad
 \|u\|_{2,m}^2=\sum_{x\in I}m(x)u(x)^2 .
\]
Each unoriented edge contributes once to the energy. We write $d_ef=f(x)-f(y)$ for an edge $e$ with an arbitrarily chosen orientation $xy$. The min--max values of $\mathcal E(u)/\|u\|_{2,m}^2$ are the eigenvalues of $\mathcal L=1-P$, where $P$ is the transition matrix restricted to $I$ as in Woess~\cite[Section~2.A, (2.6), p.~17]{Woess2000}. On either graph $M\le L\Area(\Delta)$. For the vertex walk, the degree of an interior vertex is the number of face corners at that vertex, counted with multiplicity.

\begin{proof}[Upper bounds in Theorem~\ref{thm:intro-quadratic}]
Let $\Delta$ be a least-area filling of $w_n$. Here we use the area bound: on either graph, $M\le LKn^2$. For $j\le cn^2$, we fit $2j$ squares in $[r/4,3r/4]^2$, each of side comparable to $s=n/\sqrt j$, leaving gaps of the same order. We take $c$ small enough that both the squares and the gaps are of width at least $16L$.

For each walk separately, at least $j$ of the squares have $\sum m(x)\le M/j$, where the sum is over the vertices $x\in I$ with coordinates in that square. On each of these squares we take a cutoff equal to one on the concentric square of half the side length, zero off the larger square, and with Lipschitz constant $C/s$. Let $f$ be its pullback by the coordinate map. The smaller square contains the coordinates of at least $cs^2$ interior vertices for the vertex walk, or of at least $cs^2$ faces for the face walk. For the vertex walk the cutoff vanishes on the boundary vertices. For the face walk it vanishes on every face with a side on the outer boundary, and we set its value at the exterior vertex to zero. On each edge where the cutoff varies, the assigned coordinates differ by at most $2L$, and at least one endpoint belongs to its support. Thus
\[
  \|f\|_{2,m}^2\ge cs^2,\qquad
  \mathcal E(f)\le C s^{-2}M/j\le C',
  \qquad \frac{\mathcal E(f)}{\|f\|_{2,m}^2}\le C''\frac{j}{n^2}.
\]
The gaps between the squares exceed the largest change of coordinates along an edge, so no edge joins two different supports. The same bound for the Rayleigh quotient holds throughout the span of these $j$ functions. This gives the upper bounds for both walks.
\end{proof}

For the word $w_m$ in $\Z^2$, we truncate the two usual coordinates to $[0,m]$. They change by a bounded amount along each generator of the fixed presentation. The boundary runs once around the square, so every filling has at least $cm^2$ vertices at distance at least $cm$ from the boundary and at least $cm^2$ faces at distance at least $cm$ from the exterior vertex in the dual graph. For fillings with $O(m^2)$ faces, the same test functions have Rayleigh quotients $O(j/m^2)$ for $1\le j\le cm^2$. Every least-area filling has this area bound, since we can tile the square with copies of a fixed filling of the unit square.

For the lower bounds for the eigenvalues we can work on any least-area diagram. A disk subdiagram $\Omega$ is itself least-area for its boundary word~\cite[Section~1, p.~1105]{Bridson1999}, since otherwise we could replace it by a smaller filling. In particular, $\Area(\Omega)\le\delta_{\mathcal P}(|\partial\Omega|)$. For $U\subseteq I$ we write $m(U)=\sum_{x\in U}m(x)$ and let $\partial_E U$ be the set of edges with exactly one endpoint in $U$, including edges from $U$ to the absorbing set.

\begin{lemma}\label{lem:set-iso}
Suppose $\delta_{\mathcal P}(t)\le Kt^\alpha$, $\alpha\ge1$. On either graph of a minimal diagram,
\[
               m(U)^{1/\alpha}\le C|\partial_E U|\qquad(U\subseteq I).
\]
\end{lemma}

\begin{proof}
We take a set $U$ of faces connected through shared sides. We adjoin the faces in each bounded component of the complement of their union. These components lie in the original diagram, whose complement is connected. We claim that the resulting subdiagram $\Omega$ is a closed disk. Its faces are connected through shared sides: from any added face, an arc to a side shared with a face of $U$ can be chosen to cross edges and avoid vertices. The complement of $\Omega$ is connected. If the boundary path visited a vertex twice, we could join two exterior sectors there by an arc in the complement to form a Jordan curve that meets $\Omega$ only at that vertex. The curve would separate incident faces on its two sides, contrary to their being connected through shared sides. No edge of $\Omega$ has the exterior on both sides, since $\Omega$ is a union of closed faces. Thus its outer boundary is simple (with a monogon or digon allowed).

Each edge of this boundary has a face of $U$ on one side and belonged to $\partial_E U$ before we filled the holes. By minimality,
\[
 |U|\le\Area(\Omega)\le K|\partial\Omega|^\alpha
                       \le K|\partial_EU|^\alpha .
\]
Since $m(U)\le L|U|$, this gives the claim for faces connected through shared sides. For an arbitrary set $U$ of faces, let $U_i$ be the components of the subgraph induced by $U$ in the dual graph. No dual edge joins distinct components, so $|\partial_E U|=\sum_i|\partial_E U_i|$. We also have $(\sum_i m(U_i))^{1/\alpha}\le\sum_i m(U_i)^{1/\alpha}$.

For a set $U$ of interior vertices, let $\Sigma$ be the faces that meet $U$. Counting corners gives $m(U)\le L|\Sigma|$. A side separating $\Sigma$ from the other faces or from the exterior has both ends outside $U$. The incident face in $\Sigma$ also meets $U$, so its boundary path traverses edges of $\partial_E U$ in the $1$-skeleton at least twice, counting repeated traversals. Each edge of $\partial_E U$ occurs at most twice altogether in the face boundaries, and each face has at most $L$ sides. Thus the edge boundary of $\Sigma$ in the dual graph satisfies $|\partial_E \Sigma|\le L|\partial_E U|$. Applying the estimate to $\Sigma$ proves the vertex bound, with no restriction on the vertex degrees.
\end{proof}

We write $\nu_j$ for the $j$th eigenvalue of either walk. When $I$ is nonempty, set $h_D=\min_{\varnothing\ne U\subseteq I}|\partial_E U|/m(U)$. The coarea formula for $|u|^2$ and the Cauchy--Schwarz inequality give
\[
 h_D\|u\|_{2,m}^2\le\sum_{xy}\big||u(x)|^2-|u(y)|^2\big|
                  \le\sqrt{2\mathcal E(u)}\,\|u\|_{2,m}.
\]
Thus $\nu_1\ge h_D^2/2$, the Dirichlet Cheeger inequality. For the classical graph inequalities, see Dodziuk~\cite{Dodziuk1984} and Alon and Milman~\cite{AlonMilman1985}. By Lemma~\ref{lem:set-iso}, a linear isoperimetric inequality gives a uniform positive lower bound for the first eigenvalue. We shall also use the quartic case,
\begin{equation}\label{eq:cheeger-quartic}
                     \nu_1\ge c M^{-3/2}.
\end{equation}

\begin{lemma}\label{lem:quadratic-nash}
Suppose $\delta_{\mathcal P}(t)\le Kt^\alpha$ for a fixed $\alpha>1$, and put $d=\alpha/(\alpha-1)$. On either graph of a minimal diagram, provided $I$ is nonempty,
\begin{equation}\label{eq:quadratic-nash-eigenvalues}
 \nu_j\ge c(j/M)^{2/d}\quad(1\le j\le |I|),
 \qquad T(\Delta)\le C M^{2/d}.
\end{equation}
In particular, under a quadratic isoperimetric inequality, $\nu_j\ge cj/M\ge c'j/\Area(\Delta)$ and $T(\Delta)\le CM$.
\end{lemma}

\begin{proof}
We first use the degree measure of the support to bound the energy from below. If $U\subseteq I$ and $\varnothing\ne A\subseteq U$, then $|\partial_E A|/m(A)\ge c m(A)^{-1/d}\ge c m(U)^{-1/d}$ by Lemma~\ref{lem:set-iso}. We apply the Dirichlet Cheeger inequality to functions vanishing outside $U$, keeping the original degrees. This gives
\[
             \mathcal E(v)\ge c m(U)^{-2/d}\|v\|_2^2
             \qquad(\operatorname{supp}v\subseteq U).
\]
All the norms use the degree measure. For $u\ne0$ we take $b=\|u\|_2^2/(4\|u\|_1)$ and $v=(|u|-b)_+$. Its support has mass at most $\|u\|_1/b$, while $\|v\|_2^2\ge \|u\|_2^2-2b\|u\|_1=\|u\|_2^2/2$. Truncation does not increase energy. Substitution gives the Nash inequality
\[
       \|u\|_2^{2+4/d}\le C\mathcal E(u)\,\|u\|_1^{4/d}.
\]

The next step is Nash's argument for parabolic equations in Euclidean space~\cite[pp.~935--936]{Nash1958}; we use the same argument on the graph~\cite{Coulhon1996,DiaconisSaloffCoste1996}. The heat semigroup $H_t=e^{-t\mathcal L}$ of the killed walk is self-adjoint and is a contraction on $L^1(m)$. For $\|u\|_1=1$ and $h(t)=\|H_tu\|_2^2$, the Nash inequality gives $h'(t)=-2\mathcal E(H_tu)\le-ch(t)^{1+2/d}$. Integrating this inequality, we obtain $\|H_t\|_{1\to2}\le Ct^{-d/4}$. By self-adjointness and the semigroup property, $\|H_t\|_{1\to\infty}\le Ct^{-d/2}$. We apply this bound to the indicator function of $\{x\}$ to get $(H_t\mathbf 1_{\{x\}})(x)\le Cm(x)t^{-d/2}$. Summing over $x$, we find
\[
  j e^{-t\nu_j}\le\sum_i e^{-t\nu_i}
       =\operatorname{tr} H_t\le CMt^{-d/2}.
\]
At $t=1/\nu_j$ this gives the eigenvalue bound.

We take the steps of the discrete walk at the jump times of an independent rate-one Poisson process. The probability that the walk started at $x\in I$ has not stopped by time $t$ is $H_t1(x)\le CMt^{-d/2}$. At $t_0=(2CM)^{2/d}$ this probability is at most $1/2$ for every $x\in I$. Applying the Markov property on successive intervals of length $t_0$, we find that the probability of not having stopped by time $kt_0$ is at most $2^{-k}$. Hence the mean time to absorption is at most $2t_0$. The waiting times have mean one and are independent of the discrete walk, so the expected number of steps to absorption has the same bound. This counts steps that do not change the position of the walk.
\end{proof}

\begin{proof}[Completion of the proof of Theorem~\ref{thm:intro-quadratic}]
For every minimal diagram of perimeter at most $n$ and for either walk, $M\le LK n^2$. Lemma~\ref{lem:quadratic-nash} gives the lower bounds for the eigenvalues and the upper bound for the expected absorption time. For either walk with $I$ nonempty, $T=\|\mathcal L^{-1}1\|_\infty\ge1/\nu_1$, since the spectral radius of the nonnegative Green matrix is at most its largest row sum. For the word in Proposition~\ref{prop:all-scale-square}, we proved that $\nu_1\le C/n^2$ on every minimal filling. Hence each such filling has $T\ge cn^2$. The uniform bounds and the bounds for $w_n$ give the stated orders of $\Lambda_{\mathcal P}^{(j)}$, $\Lambda_{\mathcal P}^{*,(j)}$ and the supremum of $T$ for either walk.
\end{proof}

\section{Without a quadratic isoperimetric inequality}\label{sec:universal}

If $G$ is not hyperbolic, then for each sufficiently large $n$ we choose a word for which the upper bound on the face eigenvalues in the next theorem holds on every filling. To get a sharper bound for the first face eigenvalue on a least-area filling, we cut out a smaller disk and allow its boundary word to differ from the original one.

\begin{theorem}[First eigenvalues and a bound on all fillings]\label{thm:intro-A}
For every finite presentation $\mathcal P$ of a group $G$,
\[
 G\text{ is word-hyperbolic}
 \quad\Longleftrightarrow\quad\inf_n\Lambda_{\mathcal P}(n)>0
 \quad\Longleftrightarrow\quad\inf_n\Lambda^*_{\mathcal P}(n)>0 .
\]
If $G$ is non-hyperbolic, then for every sufficiently large $n$,
$\Lambda_{\mathcal P}(n)\le C(\log n)^2/n^2$ and
$\Lambda^*_{\mathcal P}(n)\le C/n^2$.
For each such $n$ there is a cyclically reduced null word of length at most $n$ such that \emph{every} filling satisfies $\mu_j\le C(j+(\log n)^2)/n^2$ for $1\le j\le cn^2$.
\end{theorem}

The square gives many interior vertices far from the boundary and many faces far from the exterior. But for either walk the sum of degrees over all of $I$ may be much larger than $n^2$. Suppose a nonempty set $A\subseteq I$ has mass $m(A)=s_0$ and lies at distance at least $r>0$ from the absorbing set. With $M=m(I)$, we set
\[
 t=r^{-1}\log(2\sqrt{M/s_0}),\qquad
 v(x)=e^{-t d(x,A)},\qquad f=(v-e^{-tr})_+.
\]
Then $0\le f\le1$, $f=0$ on the absorbing set, $f\ge1/2$ on $A$, and
$\|f\|_{2,m}\ge\|v\|_{2,m}-e^{-tr}\sqrt M\ge\tfrac12\|v\|_{2,m}$. If $f(x)>f(y)$ on an edge $xy$, then $f(x)-f(y)\le t v(x)$. Hence
\begin{equation}\label{eq:deep-cutoff}
      \frac{\mathcal E(f)}{\|f\|_{2,m}^2}
           \le \frac4{r^2}\log^2(2\sqrt{M/s_0}).
\end{equation}

For the bound on $\Lambda_{\mathcal P}$ in Theorem~\ref{thm:intro-A}, we use the word $w_n$ of Proposition~\ref{prop:all-scale-square}. The outer boundary of every filling is simple. Among the minimal diagrams of $w_n$ we choose one that minimizes the sum of the degrees of the boundary vertices. There are at most $2|S|$ edges from a given boundary vertex to interior vertices. Indeed, suppose that two such edges issuing from the same boundary vertex have the same oriented label $a$. They cannot have the same terminal vertex, since the two edges would enclose a disk subdiagram of positive area whose boundary reads $a^{-1}a$, contrary to minimality.

For distinct interior endpoints, we use the diamond move of Gersten and Riley~\cite{GerstenRileyRadii2002}. Cutting along the two-edge path between the interior endpoints produces four labelled sides. Of the two ways of gluing adjacent inverse sides in pairs, one restores the old diagram and the other identifies the interior endpoints and splits the common initial vertex. Only one of the two sectors between the original edges contains the corner of the exterior face at that vertex. Thus one of the two new vertices is at that corner and stays on the boundary; the other, incident to at least one edge, is now interior. Neither the boundary word nor any face word changes, and no face is removed. The new diagram is still minimal, but the sum of the degrees of the boundary vertices is smaller. This rules out the repeated label.

Thus there are at most $Cn$ edges from the interior to the boundary. If $M>n^3$, the indicator function of the set of interior vertices gives $\lambda_1\le Cn/M\le C/n^2$. Otherwise the set of vertices far from the boundary in Proposition~\ref{prop:all-scale-square} has $s_0\ge cn^2$ and $r\ge cn$, and \eqref{eq:deep-cutoff} gives $\lambda_1\le C(\log n)^2/n^2$.

\subsection{Test functions on the dual graph}
For the ordinary Laplacian of a planar graph on $v$ vertices of bounded degree, Kelner, Lee, Price and Teng obtained an upper bound $Cj/v$ for the $j$th eigenvalue~\cite[Theorem~1.1, p.~1118]{KelnerLeePriceTeng2011}. For the normalized Laplacian with edge weights, see Pl\"umer~\cite[Theorem~4.5, p.~1884]{Plumer2021}. The vertices of the dual graph other than the exterior have bounded degree. Circle packing gives test functions with bounded energy when only edges between faces are counted. To control the energy including edges to the exterior vertex, we start with a function $f$ vanishing there. We give the disk for the dual vertex $x$ weight $m(x)f(x)^2$ and multiply the resulting test functions by $f$.

\begin{lemma}\label{lem:weighted-annuli}
Let $I$ be the set of vertices outside the absorbing set. Suppose the subgraph induced by $I$ is finite and planar, and the degrees $m(x)$ in the full graph satisfy $1\le m(x)\le L$ for $x\in I$. If $0\le f\le1$ vanishes on the absorbing set and $V=\sum_{x\in I} m(x)f(x)^2>0$, then
\begin{equation}\label{eq:weighted-planar}
       \nu_j\le C_L\,(j+\mathcal E(f))/V\qquad(1\le j\le c_L V).
\end{equation}
\end{lemma}

\begin{proof}
We first construct cutoffs on the graph induced by $I$. Suppose its vertices have weights $0\le\omega\le L$ of total weight $W$. For $k\le a_LW$ there are $k$ functions $0\le\eta_i\le1$ of disjoint supports, each equal to one on a set of weight at least $cW/k$, with support of weight at most $W/k$, and with energy at most $C_L$ over edges with both endpoints in $I$. We denote this part of the energy by $\mathcal E_{\rm int}$.

By the Koebe--Andreev--Thurston circle-packing theorem, the underlying simple planar graph admits a packing by disks with disjoint interiors, with disks tangent whenever their vertices are adjacent~\cite[Theorem~3.5, p.~34]{Nachmias2020}; distinct components can be placed apart. We place each vertex weight at the centre of its disk. To apply the annulus theorem of Grigor'yan, Netrusov and Yau~\cite[Theorem~3.5]{GrigoryanNetrusovYau2004}, we choose $b=\kappa W/k$ with a sufficiently small absolute constant $\kappa$. Taking $a_L$ small ensures that each atom, being at most $L$, has mass at most $b$. At sufficiently small radius every ball contains at most one centre. Successive doubling therefore gives a radius at which some ball first has mass greater than $b$, while every ball of half the radius has mass at most $b$. This verifies the scale condition of the annulus theorem. With the Euclidean covering constant it supplies $2k$ annuli, each of mass at least $b$, whose concentric enlargements from radii $(u,v)$ to $(u/2,2v)$ are disjoint.

For an annulus of radii $(u,v)$ our cutoff is one there and zero off $(u/2,2v)$, with linear interpolation on the intervening radial intervals and no inner interval if $u=0$. We evaluate it at the centres. This annular cutoff is the difference of two ball cutoffs, with the inner cutoff omitted when $u=0$. To bound its energy, we first consider a ball cutoff equal to one on $B(z,s)$ and zero off $B(z,2s)$, with Lipschitz constant $1/s$. If the cutoff takes different values at the centres of two tangent disks of radii $r_x,r_y\le s$, both centres are in $B(z,4s)$. The squared difference is at most $2(r_x^2+r_y^2)/s^2$. All these disks lie in $B(z,5s)$ and have disjoint interiors. Each disk is counted at most $L$ times, since its vertex has degree at most $L$. For any other edge on which the cutoff varies, a disk of radius at least $s$ meets $B(z,3s)$: if the nearer centre belongs to a small disk, their point of tangency lies in this ball. Such a disk contains a disk of radius $s/2$ inside $B(z,5s)$, so only a bounded number occur. The corresponding vertices are incident to $O(L)$ edges in all. The same energy bound, up to a constant, follows for the annular cutoffs. Of the $2k$ disjoint supports at least $k$ have weight at most $W/k$.

We now apply this construction with $\omega(x)=m(x)f(x)^2$ and $k=2j$. The products $f\eta_i$, extended by zero to the absorbing set, have squared norms at least $cV/j$. If $U_i$ is the support of $\eta_i$, we expand $d_e(f\eta_i)$ by the product rule on each edge with both endpoints in $I$. This gives
\[
 \mathcal E(f\eta_i)\le 2\mathcal E_{\rm int}(\eta_i)
                   +2\sum_{e\text{ incident to }U_i}|d_ef|^2 .
\]
The second sum includes edges to the absorbing set. Each endpoint belongs to at most one of the supports $U_i$, so an edge occurs in at most two of these sums. At least $j$ of the products therefore have energy at most $C_L+C\mathcal E(f)/j$. The products are orthogonal, and the energy of a linear combination is at most twice the sum of the individual energies times the squared coefficients (there are at most two terms on each edge). The min--max principle gives \eqref{eq:weighted-planar}.
\end{proof}

We apply the lemma to any filling of $w_n$, writing $a=\Area(\Delta)\ge cn^2$. Applying \eqref{eq:deep-cutoff} with $A$ the set of faces at distance at least $cn$ from the exterior gives a test function with $V\ge cn^2$ and Rayleigh quotient at most $C[1+\log(1+a/n^2)]^2/n^2$, since $M\le La$. For $a>n^3$ we take $f=1$ on the faces and $f=0$ at the exterior. Then $V\ge a$ and $\mathcal E(f)\le n$, since only sides between a face and the exterior contribute. Thus in both ranges, with $D$ as in Theorem~\ref{thm:intro-optimized},
\begin{equation}\label{eq:face-area-index}
 \mu_j(\Delta)\le\frac{C}{n^2}\big(j+D(n,a)^2\big),
                 \qquad 1\le j\le cn^2.
\end{equation}
Indeed, in the second range $(j+n)/a\le (j+1)/n^2$. The same estimate holds for the boundary of the $m$-by-$m$ square in $\Z^2$, with $n=4m$, by the coordinate count in Section~\ref{sec:minimum}. This proves the bound on every filling in Theorem~\ref{thm:intro-A}. The dependence on $a$ will be needed for Theorem~\ref{thm:intro-optimized}.

\subsection{Cutting out a smaller disk}
To remove the logarithm from the bound for the first face eigenvalue of a \emph{minimal} filling, we allow the word to change. The idea is to find a function supported away from the outer boundary of a much larger minimal diagram, with bounded energy and squared norm comparable to $r^2$. We take a component of a superlevel set and fill its holes to obtain a disk subdiagram of perimeter at most $Cr$.

For a large integer $r$, let $N=\lceil A_0r^2\rceil$ with $A_0$ fixed and sufficiently large. We choose a least-area filling of the word given by Proposition~\ref{prop:all-scale-square} at scale $N$. Its perimeter satisfies $p\le N$, and the sum of the face degrees satisfies $M\ge cN^2\ge4pr^2$. With $k=\lfloor M/r^2\rfloor$ we have $k>p$ and $M/k\asymp r^2$. The cutoff construction in Lemma~\ref{lem:weighted-annuli}, with weights given by the face degrees in the original diagram, applies since $k/M\le r^{-2}$. At most $p$ of the disjoint supports contain a face with a side on the outer boundary. One cutoff $\eta$ therefore has Dirichlet energy at most $C_L$, is one on a set of mass at least $cr^2$, and has support of mass at most $Cr^2$.

We extend $\eta$ by zero to the exterior. At most $Cr^2$ edges have endpoints at which $\eta$ takes different values. The coarea formula and the Cauchy--Schwarz inequality give
\[
 \int_{1/4}^{1/2}|\partial_E\{\eta>t\}|\,dt\le\sum_e|d_e\eta|\le Cr.
\]
For some $t\in[1/4,1/2]$ the set $\{\eta>t\}$ has edge boundary of size at most $Cr$. The function $(\eta-t)_+$ has energy $\le C_L$ and squared norm $\ge cr^2$. On some component $U$ of the subgraph induced by $\{\eta>t\}$ in the dual graph, the restriction, extended by zero outside $U$, has Rayleigh quotient at most $C/r^2$, since squared norms and energies add over the components. The edge boundary of $U$ has size at most $Cr$. Filling its holes as in Lemma~\ref{lem:set-iso} gives a disk subdiagram; it is still minimal because it lies in a minimal diagram. Its boundary is simple and has length at most $Cr$. We extend the function on $U$ by zero to the added faces. Its norm and energy do not change: the degrees of the faces in $U$ are unchanged and the value on the other side of each edge leaving $U$ is still zero.

If two consecutive boundary letters cancel, we fold the corresponding sides together. This can also be done across the basepoint. When the boundary length is at least three, the endpoints away from the shared vertex are distinct. Identifying the sides makes their shared vertex interior and merges their other endpoints; the remaining boundary is again a circle and the quotient is a disk. A boundary of two sides with inverse labels is impossible, since the disk has positive area and is minimal. Folding does not change the faces or their degrees. The resulting diagram is still least-area, since free reduction preserves filling area. Let $u,v\ge0$ be the test values on the faces incident to the two sides. The boundary terms $u^2+v^2$ are replaced by $(u-v)^2\le u^2+v^2$. If both sides belong to the same face, the new term is zero. We repeat until the boundary word is cyclically reduced. This leaves a minimal diagram of perimeter at most $Cr$ and first face eigenvalue at most $C/r^2$. Choosing $r=\lfloor n/C'\rfloor$ for a sufficiently large fixed $C'$ makes its perimeter at most $n$.

This proves the bounds for both least-area profiles at $j=1$ in the non-hyperbolic case. A word-hyperbolic group satisfies a linear isoperimetric inequality in every finite presentation, by Gromov~\cite[Corollary~2.3.B, p.~103]{GromovHyperbolic1987}. Lemma~\ref{lem:set-iso} and the Cheeger estimate give uniform positive lower bounds for both first eigenvalues. This completes the proof of Theorem~\ref{thm:intro-A}.

Can the bound for $\Lambda_{\mathcal P}$ in Theorem~\ref{thm:intro-A} be improved to $C/n^2$ for every finite presentation of a non-hyperbolic group? The bounded-degree argument for faces does not carry over directly. Even for a non-hyperbolic group with triangular relators, a vertex in a least-area diagram can have arbitrarily large degree. The presentation
\[
 \langle a,b,t,x,y,z\mid atb,\ b^{-1}ta^{-1},\ xyz^{-1},\ yxz^{-1}\rangle
\]
gives a group with $\Z^2$ as a free factor, so it is not hyperbolic. The word $t^{2m}$ bounds a fan of $2m$ triangles in which the edges directed out of the central vertex are labelled alternately $a,b^{-1}$ in cyclic order. Each oriented relator has $t$-exponent sum of absolute value at most one, so every filling of that word has at least $2m$ faces. The fan is therefore least-area and its central vertex has degree $2m$. This shows that degrees cannot be bounded uniformly over all least-area diagrams; it does not rule out choosing a different least-area diagram with bounded degrees.

\section{Choosing a filling}\label{sec:optimized}

For the lower bound on $\Theta_{\mathcal P,F}^{(j)}(n)$ in Theorem~\ref{thm:intro-optimized}, we find a null word of length at most $n$ for which every filling of area at most $F$ has a small $j$th face eigenvalue. For the upper bound, we choose a filling of area at most $F$ for each null word of length at most $n$.

For the lower bound we take the word from Proposition~\ref{prop:all-scale-square}. Every filling has at least $cn^2$ faces, so the $j$th face eigenvalue exists throughout the required range. For a filling of area $a\le F$, we have $D(n,a)\le D(n,F)$, and the reciprocal of the bound in \eqref{eq:face-area-index} gives the required lower bound.

The distance to the exterior also gives a lower bound for the expected absorption time. Let $b$ be the exterior vertex and $\tau_b$ the first time the face walk reaches it. If a filling of perimeter at most $n$ has a face $x$ at distance at least $c_0 n$ from $b$ in the dual graph, then
\begin{equation}\label{eq:deep-face-time}
 \E_x\tau_b\ge c\,n^2/\log n
\end{equation}
for large $n$. Indeed, let $X_j$ be the walk on the full dual graph with the loops at $b$ removed. Only transitions out of $b$ change, so the law of the first visit to $b$ is unchanged. The degree of $b$ is now at most $n$. By the Varopoulos--Carne bound~\cite[Theorem~13.4, p.~429]{LyonsPeres}, the probability of a visit to $b$ by time $t$ satisfies
\[
 \Pr_x(\tau_b\le t)\le\sum_{j=1}^{t}\Pr_x(X_j=b)
       \le 2t\sqrt n\exp\!\left(-\frac{c_0^2 n^2}{2t}\right).
\]
Taking $t=\lfloor c_1 n^2/\log n\rfloor$ with $c_1$ small gives $\E_x\tau_b\ge t/2$. Proposition~\ref{prop:all-scale-square} gives such a face on every filling of the chosen word. For the upper bounds we start with any null word and build a filling within the area bound.

\subsection{A disk with a large first face eigenvalue}

We first construct the dual graph of the disk, including the vertex for the exterior. Most of the graph consists of concentric cycles arranged in blocks. The cycles in a block have the same number of vertices, and this number doubles from one block to the next. We fix $N=2^k$ with integer $k\ge6$, choose an integer $1\le h\le k$, and take $s$ to be the largest power of two at most $N/h$. Thus $N/2<hs\le N$. Between the centre and the boundary there are $h$ blocks of $s$ consecutive levels each. We shall see that the reciprocal of the first eigenvalue is bounded by $Cs^2$, although the distance from the centre to the boundary is of order $hs$. For related constructions for the ordinary graph Laplacian, see Louder and Souto~\cite[Theorems~1.1--1.2, p.~1]{LouderSouto2012} and Lee and Qin~\cite[Section~2, Theorem~1, pp.~2--3]{LeeQin2012}.

The levels other than the centre are concentric cycles about a central vertex. We start with a four-cycle and double their lengths until the first $s$-cycle, which is the initial level of block $0$. For $0\le\ell<h$, block $\ell$ consists of $s$ cycles, each of length $W_\ell=2^\ell s$. After the last block we halve the cycle length at each step until it is $N$. This is possible since $s=N/2^{\lceil\log_2h\rceil}$ and $h-1\ge\lceil\log_2h\rceil$; if $W_{h-1}=N$ we add no further cycles.

We retain the edges of every cycle. Two cycles of the same length are joined by matching their vertices in circular order. At a doubling, the vertices of the larger cycle come in consecutive pairs. Both vertices of each pair are joined to the corresponding vertex of the smaller cycle. A halving uses the same rule with the cycles reversed. The centre is joined to the four vertices of the first cycle. In the complementary disk on the sphere we join the $N$ vertices of the last cycle to a single vertex $b$. We denote the resulting graph on the sphere by $\mathcal H$. Its faces are triangles or quadrilaterals, and all vertices except $b$ have degree at most six.

We now pass to the dual graph on the sphere. The dual face corresponding to $b$ has a simple boundary of length $N$: its boundary visits the dual vertices of the $N$ distinct triangles at $b$, in order. We take this face as the exterior. The remaining disk $D_{N,h}$ has one face for each vertex of $\mathcal H$ other than $b$. The face walk follows the corresponding edges of $\mathcal H$ until it reaches $b$.

We index the levels in $\mathcal H$ by $0,\ldots,J$, starting at the centre. For $i\ge1$ let $N_i$ be the length of the cycle on level $i$, and set $N_0=1$ for the central vertex. The blocks account for $hs$ levels and for $s\sum_{\ell=0}^{h-1}W_\ell=(2^h-1)s^2$ vertices. There are $O(\log s+h)$ initial and final levels, with a total of $O(s+W_{h-1})$ vertices. Thus
\[
 J\asymp hs\asymp N,\qquad \Area(D_{N,h})=\sum_{i=0}^JN_i\asymp 2^h s^2.
\]
We will also need $N_i\ge c(i+1)$ when we label the disk. This holds on the initial doubling levels. In block $\ell$ the index is at most $\log_2 s+(\ell+1)s$ while $N_i=2^\ell s$. On the final levels the claim follows from $N_i\ge N$ and $J\le CN$.

\emph{The radial walk.} Let $m_i$ be the sum of the degrees on level $i$, and let $c_i$ count the edges from level $i$ to the next level, setting $c_J=N$ for the last edges to $b$. Both quantities are comparable to $N_i$. We project the walk on $\mathcal H$ to the sequence of levels it visits. All vertices on level $i$ have the same number of neighbours on each adjacent level, so this is a birth--death chain: at $i$ it moves outward with probability $c_i/m_i$, moves inward with probability $c_{i-1}/m_i$ (where $c_{-1}=0$), and otherwise remains at $i$. The positive first eigenfunction of this killed chain lifts to a positive eigenfunction on the faces of $D_{N,h}$. Deleting $b$ from $\mathcal H$ leaves a connected graph, so by Perron--Frobenius this gives the first eigenvalue of the full face walk.

For a function $u$ constant on each level, write $u_i$ for its value on level $i$, and set $u_{J+1}=0$ at $b$. Its squared norm and energy are $\sum_{i=0}^Jm_i u_i^2$ and $\sum_{i=0}^Jc_i(u_i-u_{i+1})^2$, respectively, since $u$ is constant along the edges within each level. We set
\[
 A_i=\sum_{t=0}^i m_t,\qquad R_i=\sum_{t=i}^J c_t^{-1}.
\]
We use a discrete form of Muckenhoupt's weighted Hardy criterion~\cite{Muckenhoupt}. Since $u_i=\sum_{t=i}^J(u_t-u_{t+1})$, we apply the finite weighted Hardy inequality~\cite[Proposition~A.3, pp.~317--318]{ChenSaloffCoste2013} after reversing the level index. This gives
\[
       \mu_1(D_{N,h})^{-1}\le 4\max_{0\le i\le J} A_iR_i .
\]

Suppose level $i$ is in block $\ell$, whose cycles have length $W=2^\ell s$. The sum of the degrees in that block is at most $CsW$. The same bound holds for the sum over the earlier blocks, since the cycle lengths form a geometric progression. For $R_i$ we sum in the opposite direction:
\[
 \begin{split}
 A_i&\le C\left(s+s\sum_{q=0}^\ell 2^q s\right)\le CsW,\\
 R_i&\le C\left(\sum_{q=\ell}^{h-1}\frac{s}{2^q s}
                 +\sum_{j\ge0}\frac1{2^jN}\right)
       \le C\left(\frac{s}{W}+\frac1N\right).
 \end{split}
\]
The term $1/N$ comes from the final halving levels and the edges to $b$. Since $W\le 2^{h-1}s\le Ns$, we have $A_iR_i\le C(s^2+sW/N)\le Cs^2$.

On the initial levels the cycle lengths double, so the sum of the degrees up to a level of length $W$ is $A_i=O(W)$. The rest of the initial levels contribute $O(1/W)$ to $R_i$, while the blocks and final levels contribute $O(1)$ by the preceding sum with $\ell=0$. After the last block $R_i=O(1/N)$, and $A_i$ is bounded by the sum of all the degrees. Thus the two end regions give, respectively,
\[
     A_iR_i\le CW(1/W+1)\le C's^2,\qquad
     A_iR_i\le C2^hs^2/N\le Cs^2,
\]
using $W\le s$ in the first case and $2^h\le N$ in the second. This proves the uniform estimate
\[
                    \mu_1(D_{N,h})^{-1}\le Cs^2\le C(N/h)^2.
\]

\subsection{Filling the given word}

We return to the fixed presentation in Theorem~\ref{thm:intro-optimized}. Let $w$ be a null word of length $m\ge64$, and let $N$ be the smallest power of two at least $m$, so $m\le N<2m$. We insert one formal identity letter after each of $N-m$ distinct positions in $w$. We assign group elements to the vertices of $D_{N,h}$ and remove the inserted letters afterwards. We label the successive boundary vertices by $g_0,\ldots,g_N$, where $g_i$ is the group element represented by the first $i$ letters of the padded word. Thus $g_0=g_N=1$. It suffices that the word distance between the endpoint labels of every interior edge be uniformly bounded.

We place the vertices on each cycle of $\mathcal H$ at equal angular intervals. Matched vertices on successive cycles of equal length have the same angle. When the lengths differ, the two vertices paired with a vertex of the smaller cycle lie in its angular sector. Each face of $\mathcal H$ between levels $i$ and $i+1$ gives a vertex of $D_{N,h}$. We place it at radius $i+1/2$ and at the middle angle of the face. The vertices corresponding to the $N$ outer triangles are equally spaced at radius $R=J+1$. For two adjacent vertices of $D_{N,h}$, the radial coordinates differ by at most two and the angular separation is at most $C/N_i$, where $i$ is the smaller inner level index of the two corresponding faces of $\mathcal H$ (with index $J$ for an outer triangle). The angular displacement contributes at most $C(i+1)/N_i$ to the Euclidean distance, which is bounded since $N_i\ge c(i+1)$. We need only this bound for adjacent vertices of the disk.

Suppose first that $G$ acts on a space $X$ as in Theorem~\ref{thm:intro-optimized}, and fix $x_0\in X$. At the $N$ boundary angles we use the orbit points $g_i x_0$ of the prescribed labels. Joining successive points by geodesics gives a uniformly Lipschitz map on the circle of radius $R$, since the steps are bounded and $R\asymp N$. By the Lipschitz extension assumption and rescaling, this map extends to the Euclidean disk with a uniform Lipschitz constant. By cocompactness we can label each interior vertex by a group element whose orbit point is within a bounded distance of the image of that vertex under the extension. The boundary labels remain unchanged. The resulting orbit points at neighbouring vertices are a bounded distance apart. If $g,g'$ label adjacent vertices, then $d_X(x_0,g^{-1}g'x_0)$ is bounded. By properness, $g^{-1}g'$ belongs to a fixed finite subset of $G$, so its word length is bounded independently of $w$ and $h$.

In a CAT(0) space the geodesic cone from one point of the curve gives the required extension: at a common radius $t$ in the unit disk, the distance between the images of two points is at most $t$ times the distance between the corresponding points on the curve. Along each radius the map is Lipschitz with constant at most the diameter of the curve. For $H_{2q+1}(\Z)$, $q\ge2$, we use the lattice action on the real Heisenberg group in its Carnot--Carath\'eodory metric and the extension theorem of Wenger and Young for $J^1(\R^q)$~\cite[Theorem~1.1, p.~1138]{WengerYoung}; the real Heisenberg group is isomorphic as a graded group to this jet group with an equivalent homogeneous metric.

If $G$ is automatic, we obtain the labels without a disk extension in an ambient space. Let $\gamma_g$ denote its chosen normal-form path from $1$ to $g$, extended to be constant after it reaches $g$. Paths to adjacent endpoints stay a bounded distance apart at equal times, and their lengths differ by a bounded amount~\cite[Section~1, (1.1), and Lemma~1.3(1), pp.~754--755]{BridsonCombing2003}. By comparing the paths to successive vertices of a geodesic in the Cayley graph, we obtain $|\gamma_g|\le C(1+|g|)$. The combing generators and the fixed presentation generators give word metrics comparable up to a constant, so all paths to the boundary labels $g_i$ reach their endpoints by an integer time $t_N\le CN$.

We label all vertices of radius at most $R/2$ by $1$. Along each outward radius we follow the combing path to the label of a boundary vertex closest in angle. If $\ell$ is its index, then for $R/2\le\rho\le R$ the label is $\gamma_{g_\ell}(\lfloor t_N(2\rho/R-1)\rfloor)$. At $\rho=R$ this gives exactly the prescribed boundary label. At adjacent outer vertices the times differ by $O(1)$, since $t_N/R$ is bounded. The boundary indices also differ by $O(1)$ cyclically, since $N_i\gtrsim N$ on the outer half. At a common time, the paths to these two boundary labels are a bounded distance apart by the preceding estimate applied successively. The labels remain a bounded word distance apart when we evaluate the two paths at their respective times, since these times differ by $O(1)$. For an edge crossing the circle of radius $R/2$, the outer endpoint has bounded combing time, so its label is a bounded word distance from $1$.

In either case, to obtain a diagram over the fixed presentation we replace each oriented interior edge by a nonempty word of bounded length representing $g^{-1}g'$, where $g$ and $g'$ label its initial and terminal vertices, using inverse words in opposite orientations. We keep the padded boundary word unchanged, including the inserted identity letters. The boundary of each original face now reads a null word from a fixed finite list, since these words have bounded length in a fixed alphabet. We show that each of these words bounds a diagram homeomorphic to a disk.

The simplicial Rips complex at the radius defining the fixed presentation is simply connected. By admissibility an edge loop bounds a van Kampen diagram, and in the Rips complex the vertices of each relator loop lie in a simplex. For a null word of length at least three, we take a polygon with one side for each letter. We map its successive vertices to the successive prefixes of the word, allowing edges to collapse. Simple connectivity gives a continuous extension to the polygon, with image in a finite subcomplex by local finiteness. Relative simplicial approximation~\cite[Theorem, p.~40]{Zeeman1964} gives a triangulation on which the extension is simplicial, without subdividing the boundary. On the boundary we keep the prescribed letters and their exponents. For each other edge we choose an orientation and use the generator for $g^{-1}g'$, where $g$ and $g'$ are the labels of the initial and terminal vertices. We use the inverse letter for the reverse orientation and the identity generator for a collapsed edge. Every triangle then reads an allowed null word, even when some group elements coincide. A null word of length one or two is itself a relator.

We fill the characteristic polygon of each original face with its chosen disk before gluing the subdivided sides. Each original face is replaced by at least one face and at most a fixed number of faces. The faces of each replacing disk are connected through shared sides: an arc between two face interiors in a disk can be taken transverse to the edges and disjoint from the vertices. Finally, at each inserted identity letter $e$ we attach a triangle along the boundary edges reading $ue$, where $u$ is the preceding signed letter. On the new boundary the remaining side reads $u$. These pairs share no boundary edge, since at most one identity letter was inserted after each original letter. The resulting boundary word is exactly $w$.

\smallskip
\noindent\emph{Comparison of the first face eigenvalues.}
Let $C_v$ be the disk replacing the original face $v$, and choose a face $z_v$ of $C_v$. Given a function $f$ on the new faces, extend it by zero at the exterior vertex. Set $g(v)=f(z_v)$ and take $g=0$ at the original exterior vertex. Each face $x$ of $C_v$ can be joined to $z_v$ by a path of bounded length along dual edges internal to $C_v$, before identifying its sides. A telescoping sum and the Cauchy--Schwarz inequality bound $|f(x)-g(v)|^2$ by a constant times the sum of the energy terms on these internal edges. Each disk $C_v$ has a bounded number of faces, and every new face has degree at most three. For each added boundary triangle, its contribution to the squared norm is at most three times the energy across its edge to the exterior.

To estimate the energy of $g$, we choose a path in the new dual graph for each original side. For an interior side, we join the representatives by paths within the two replacing disks and cross one of the edges into which that side was subdivided. Only the path chosen for that original side uses this crossing edge. For a side on the original outer boundary, the path joins the representative to the exterior, through the attached boundary triangle if there is one. All paths have bounded length. Only the paths for the sides of $v$ can use edges internal to $C_v$, and at most two original boundary sides give paths through any added triangle. Thus each new edge is used a bounded number of times. Using the degree measure on each dual graph, we obtain
\begin{equation}\label{eq:lean-template-transfer}
 \|f\|_{\rm new}^2\le C\bigl(\|g\|_{\rm old}^2+\mathcal E_{\rm new}(f)\bigr),
 \qquad \mathcal E_{\rm old}(g)\le C\mathcal E_{\rm new}(f).
\end{equation}
Write $\Delta_h$ for the resulting filling. The Poincar\'e inequality on $D_{N,h}$, applied to $g$, gives $\mu_1(\Delta_h)^{-1}\le C(1+\mu_1(D_{N,h})^{-1})$. Thus
\begin{equation}\label{eq:lean-labeled-model}
 \Area(\Delta_h)\le C2^h(N/h)^2,\qquad
 \mu_1(\Delta_h)^{-1}\le C(N/h)^2.
\end{equation}
This area bound includes the $O(N)$ added boundary faces. The finitely many shorter words have fixed fillings; the empty word has a filling with no faces.

\begin{proof}[Completion of the proof of Theorem~\ref{thm:intro-optimized}]
For $F\ge An^2$ we choose
\[
 h=\max\!\left\{1,\left\lfloor\tfrac12
       \min\{\log_2 N,\log_2(1+F/N^2)\}\right\rfloor\right\},
\]
omitting the area term if $F=\infty$. Then $h\asymp D(N,F)$ and $2^h\le 2(1+F/N^2)^{1/2}$ when $F$ is finite. For finite $F$ we have $y:=F/N^2\ge A/4$, since $N<2m\le2n$. The area in \eqref{eq:lean-labeled-model} is then at most $C N^2(1+y)^{1/2}\le N^2y=F$ for $A$ sufficiently large. It follows that $\mu_1^{-1}\le Cn^2/D(n,F)^2$: indeed $x/D(x,F)$ is increasing, being the maximum of $x/(1+\log(x+2))$ and $x/(1+\log(1+F/x^2))$, and $D(2n,F)\asymp D(n,F)$. Short words are covered by enlarging the constants.

With $h=1$ we obtain a quadratic isoperimetric inequality for this presentation. A least-area filling of a word of length at most $n$ therefore has area at most $C_0n^2$, hence at most $F$ if we take $A\ge C_0$. By \eqref{eq:quadratic-nash-eigenvalues} it satisfies $\mu_j^{-1}\le Cn^2/j$. Here and below the reciprocal is zero if there are fewer than $j$ faces. On $\Delta_h$, we use $\mu_j^{-1}\le\mu_1^{-1}\le Cn^2/D(n,F)^2$. Choosing the better of the two fillings for each $j$, we obtain the upper bound for $\Theta_{\mathcal P,F}^{(j)}(n)$, since the smaller of these bounds is at most $2Cn^2/(j+D(n,F)^2)$.

For the last assertion we choose $h=\max\{1,\lfloor\eta\log_2 N\rfloor\}$ with $\eta=\min\{\varepsilon,1\}/2$. The labelled filling has area $O_\varepsilon(n^{2+\varepsilon}/\log^2n)$ and satisfies $\mu_1^{-1}=O_\varepsilon(n^2/\log^2n)$. Let $\tau_{\rm ct}$ denote the absorption time when the walk takes its steps at the jump times of an independent rate-one Poisson process, including steps that leave the current face unchanged. Write $V$ for the sum of the face degrees. For the semigroup of the killed walk we have
$\Pr_x(\tau_{\rm ct}>t)\le\min\{1,\sqrt{V/m(x)}\,e^{-t\mu_1}\}$
since $H_t$ has norm $e^{-t\mu_1}$ on $L^2(m)$. Integrating shows that the largest expected absorption time is at most $\mu_1^{-1}(1+\tfrac12\log V)$. The same bound holds for the expected number of discrete steps, since the waiting times between steps have mean one and are independent of the walk. As $V\le3\Area(\Delta_h)$, this is $O_\varepsilon(n^2/\log n)$. The lower bounds for $\mu_1^{-1}$ and $T$ hold for every filling of the word chosen at the start of the section.
\end{proof}

\subsection{Quasi-isometry invariance}

Gersten and Riley proved quasi-isometry invariance of the gallery length function in their fat presentations~\cite[Theorem~8.1, p.~412]{GerstenRileyDuality2005}. The gallery length of a diagram is the diameter of its dual graph, including the exterior vertex. As in their argument, we associate a null word in $H$ to a null word in $G$. From any filling of the word in $H$, we construct a filling of the original word in $G$. We need to control the reciprocal of the first face eigenvalue in this construction.

\begin{proposition}\label{prop:first-face-qi}
Let $G$ and $H$ be quasi-isometric finitely presented groups. Fix an admissible Rips presentation $\mathcal P_G$ for $G$ and any finite presentation $\mathcal P_H$ with a nonempty generating set for $H$. There is a constant $C$ with the following property. For each null word $w$ in the generators of $G$, there is a null word $v$ in the generators of $H$ with $|v|\le C|w|$ such that every filling $E$ of $v$ gives a filling $\Delta$ of $w$ satisfying
\[
                 \mu_1(\Delta)^{-1}\le C\bigl(1+\mu_1(E)^{-1}\bigr).
\]
Here the reciprocal is zero on a diagram without faces.
\end{proposition}

\begin{proof}
We fix a quasi-isometry $q:G\to H$ and a quasi-inverse $r:H\to G$. For a null word $w=a_1\cdots a_m$ in $G$, set $g_0=1$ and $g_i=a_1\cdots a_i$ for $1\le i\le m$. Put $h_i=q(g_i)$, so $h_m=h_0$. For each pair $h_{i-1},h_i$ we choose a nonempty word of bounded length joining them in the Cayley graph of $H$. If the two vertices agree, we use a fixed word $bb^{-1}$. Concatenating these words gives a null word $v$ of length at most $Cm$.

In any filling $E$ of $v$ we label the base vertex by $h_0$ and the other vertices by reading edge labels along paths from it. The label of a vertex is independent of the path, since $E$ is simply connected and its face words are null. For an oriented edge with initial and terminal labels $h,h'$, we replace the edge by a nonempty word of bounded length representing $r(h)^{-1}r(h')$. We use the inverse word in the reverse orientation. Only finitely many resulting face words can arise, since their lengths are bounded. We choose a disk over $\mathcal P_G$ for each word. We fill each characteristic polygon with its chosen disk before identifying the subdivided sides. The faces in each replacing disk are bounded in number and remain connected through shared sides even if the original face met the same edge twice. The marked occurrences on the boundary now have labels $r(h_i)$, which need not equal $g_i$.

Since $rq$ is a bounded distance from the identity, we choose a word $p_i$ of bounded length joining $r(h_i)$ to $g_i$. We use a single identity letter if necessary and take $p_m=p_0$. Let $\alpha_i$ be the word read along the segment of the subdivided boundary corresponding to the chosen path from $h_{i-1}$ to $h_i$. This segment has bounded length. The strip for $a_i$ has inner side reading $\alpha_i$, outer side reading $a_i$, and radial sides reading $p_{i-1}$ and $p_i$. Its boundary word $\alpha_i p_i a_i^{-1}p_{i-1}^{-1}$ has bounded length and represents the identity. Only finitely many such words can arise, and we choose a disk filling for each.

The strips follow the boundary path of $E$ in order, keeping separate occurrences of vertices and edges. A thin closed regular neighbourhood of $E$ in the plane is a disk, since it is a compact planar surface that deformation retracts onto the simply connected diagram. On the outer boundary of this neighbourhood we place one new vertex for each marked occurrence of a prefix, in cyclic order. Each radial path runs through the sector outside the diagram between the incoming and outgoing boundary sides at that visit. Different marked visits to a cut vertex have separate radial paths. The two traversals of a boundary bridge remain separate occurrences, even if they lie on the inner boundary of the same strip. Equivalently, we can draw a radial arc at each boundary corner and then omit the arcs at the intermediate corners of each segment that reads $\alpha_i$. Adjacent strips share their radial side; their outer sides, read in order, give exactly $a_1,\ldots,a_m$. If a vertex or side occurs more than once in a single strip, we make the identifications on its inner boundary only. Words of bounded length, including the empty word, may be treated separately.

In the dual graph each new face of strip $i$ can be joined to the exterior by a path of bounded length that stays in the strip until it crosses the edge labelled $a_i$. For each occurrence on the original outer boundary of a side incident to a face, we join that face's representative to the exterior. We cross one of the edges into which the side was subdivided, enter the corresponding strip, and leave it across the edge labelled $a_i$. Only occurrences on the segment that reads $\alpha_i$ give paths through strip $i$, and their number is bounded independently of the perimeter. For each original interior side, as before, we join the representatives within the two replacing disks and cross one edge of that subdivided side. Only the path for that original side uses the crossing edge. An original dual loop gives zero energy and needs no path. The norm and energy estimates~\eqref{eq:lean-template-transfer} therefore bound the reciprocal of the first face eigenvalue by $C(1+\mu_1(E)^{-1})$. If $E$ was a tree, there are only strip faces, and these paths give the constant bound directly.
\end{proof}

\begin{proof}[Proof of Theorem~\ref{thm:intro-invariant}]
If both presentations are admissible Rips presentations, take the infimum over fillings of $v$ in Proposition~\ref{prop:first-face-qi}, and then the maximum over $|w|\le n$. This gives $\Theta_G(n)\le C(1+\Theta_H(\lceil Cn\rceil))$. Interchanging $G$ and $H$ gives the reverse comparison.

Now fix any finite presentation $\mathcal P$ of $G$. If $G$ is hyperbolic, its Dehn function is bounded above by a linear function. Then Lemma~\ref{lem:set-iso} and the Cheeger inequality bound $\mu_1^{-1}$ on every least-area filling, so the infimum over all fillings is bounded as well. If $G$ is not hyperbolic, the last assertion of Theorem~\ref{thm:intro-A} at $j=1$ bounds $\mu_1$ from above on \emph{every} filling of the word chosen there. This gives the lower bound $cn^2/(\log n)^2$ for $\Theta_{\mathcal P}(n)$.
\end{proof}

\section{Groups with quartic Dehn functions}\label{sec:separation}

We prove Theorem~\ref{thm:intro-examples} by comparing fillings in two groups with quartic Dehn functions. In a reduced diagram over the presentation of the first group, we choose a path in the dual graph from each face to the exterior. We bound the lengths of these paths and the number that traverse each dual edge. For the second group there are words of length at most $n$ for which every filling has intrinsic diameter at least $cn^3$. We shall prove the following estimates for least-area fillings as well.

\begin{theorem}[Least-area fillings of the quartic examples]\label{thm:quartic-least-area}
The groups in Theorem~\ref{thm:intro-examples} admit finite presentations $\mathcal P,\mathcal Q$ with Dehn functions of order $n^4$ such that
\[
 \Lambda_{\mathcal P}(n)\asymp\Lambda^*_{\mathcal P}(n)\asymp n^{-4},
 \qquad
 \Lambda_{\mathcal Q}(n)\asymp\Lambda^*_{\mathcal Q}(n)\asymp n^{-6}.
\]
\end{theorem}

First we record the bounds that follow from the Dehn function. Suppose $\delta_{\mathcal P}(n)\asymp n^\alpha$ for some $\alpha>1$. Every minimal diagram of perimeter at most $n$ has $M\lesssim n^\alpha$, so Lemma~\ref{lem:quadratic-nash} gives a lower bound $c n^{2-2\alpha}$ for either least-area profile at $j=1$. For the upper bound on $\Lambda^*_{\mathcal P}$, we choose a word attaining $\delta_{\mathcal P}(n)$ and cyclically reduce it without changing its least area. On a minimal filling, the function equal to one on every face has squared norm at least the area and energy at most the perimeter. Thus, for large $n$,
\[
             c n^{2-2\alpha}\le \Lambda^*_{\mathcal P}(n)
                           \le C n^{1-\alpha}.
\]
At $\alpha=4$ these bounds are $cn^{-6}$ and $Cn^{-3}$. The least-area face profile $\Lambda^*_{\mathcal Q}$ attains the lower bound, whereas $\Lambda^*_{\mathcal P}$ has order $n^{-4}$. Theorem~\ref{thm:intro-A} also gives $\Lambda^*_{\mathcal P}(n)\le Cn^{-2}$, which improves the upper bound when $\alpha<3$.

We shall need just one comparison between the two walks. If the lengths of the face boundaries are at most $L$ and there is an interior vertex, then
\begin{equation}\label{eq:face-averaging}
 \mu_1(\Delta)\le C_L\lambda_1(\Delta).
\end{equation}
For a vertex function $u$ vanishing on the boundary, let $Au(F)$ be the mean of $u$ over the corners of $F$. We write $S$ for the sum of $(u(v)-Au(F))^2$ over all corners, counting each occurrence of a vertex separately. Since each interior vertex occurs as many times as its degree, the variance identity on each face gives
\[
 \|Au\|_{2,m^*}^2=\|u\|_{2,m}^2-S,\qquad
 S\le 2L^2\mathcal E(u),
\]
where $m^*$ is the face degree measure. For the inequality, the difference between any two corner values telescopes along at most $L$ sides, and each edge occurs in the face boundaries at most twice, counting repetitions. If two faces share an edge, we subtract the value of $u$ at one endpoint from both means. On an edge between a face and the exterior, the endpoint value is zero, so we compare the face mean directly with zero. Each corner occurs in at most two of these comparisons, whence $\mathcal E^*(Au)\le4S$ for the face energy $\mathcal E^*$. If $\lambda_1\le(4L^2)^{-1}$, the average of a first vertex eigenfunction retains at least half its squared norm. Its Rayleigh quotient is at most $C_L\lambda_1$, which proves \eqref{eq:face-averaging}. Otherwise $\mu_1\le1$, by testing with the indicator function of the set of faces, and this suffices. Thus a lower bound for $\mu_1$ implies one for $\lambda_1$, and an upper bound for $\lambda_1$ implies one for $\mu_1$.

\begin{proof}[Proof of Theorems~\ref{thm:intro-examples} and~\ref{thm:quartic-least-area}]

\emph{The double.}
Let $\phi$ be the automorphism of the free group $F(a,b,c)$ given by $\phi(a)=a$, $\phi(b)=ba$ and $\phi(c)=cb$. We use the presentation
\begin{equation}\label{eq:quartic-presentation}
 \mathcal P=\langle a,b,c,s,t\mid u^{-1}xu\,\phi(x)^{-1}
                       \ (u\in\{s,t\},\ x\in\{a,b,c\})\rangle.
\end{equation}
This is the double of the mapping torus over its free base. The base embeds because the group is a semidirect product of $F(a,b,c)$ with $F(s,t)$. In the basis $s,st^{-1}$ of $F(s,t)$ the second generator centralizes the base. With $(a_1,a_2,a_3)=(c,b,a)$ this is $B_3$ in the notation of Bridson and Riley~\cite[Section~3.2, p.~128]{BridsonRiley2009}. Bridson proved that $B_m$ has Dehn function of order $n^{m+1}$~\cite[Section~3.4, p.~452]{BridsonPolynomial2002}. For general bounds on the Dehn functions of such doubles, see Brady and Soroko~\cite[Propositions~3.3--3.4, p.~9]{BradySoroko2019}. Here the spectral estimates require control of the diagrams over the particular presentation \eqref{eq:quartic-presentation}.

\smallskip
\noindent\emph{Paths from the faces to the exterior.}
A face with boundary word $u^{-1}xu\phi(x)^{-1}$ is of type $x$. We call the edge labelled $x$ its bottom edge and the path labelled $\phi(x)$ its top path. The other two sides, labelled $u$, are the stable sides.

We map the diagram to the Cayley tree of $F(s,t)$, sending each base edge to a vertex. The arcs joining the midpoints of the stable sides continue across common stable edges to form bands, each lying over one tree edge. In a reduced diagram, consecutive letters on the bottom path of a band cannot be inverses: the corresponding faces would share the same stable side of the same relator with opposite orientations, a cancellable pair. The bottom path of a closed band would therefore read a nonempty cyclically reduced word. It labels a loop in the diagram and hence represents the identity, which is impossible because the free base embeds. Both ends of every band therefore lie on the outer boundary. In particular each tree edge crossed by any face is traversed by the projected boundary. For a related use of bands and trapezia in a cubic example, see Ol'shanskii and Sapir~\cite[Lemma~2.1, p.~3, and the discussion after Remark~2.3, p.~4]{OlshanskiiSapir2005}.

As in Bridson and Groves~\cite[Definition~1.5.1]{BridsonGroves2010}, we draw a second arc in each face, joining the bottom $x$-edge to the unique $x$-edge on the top path. We join these arcs along shared sides in faces of the same type to form chains. A chain traverses one tree edge per face. If a step reversed the preceding one, the two faces would have the same stable letter and meet either along their bottom edges or along the two distinguished top edges. In either case the occurrence of the common letter in the relator is the same, so the pair is cancellable. The two ends of the arc in one face cannot be identified, since the bottom and top of the face lie over distinct tree vertices. Thus the chain projects to a nonbacktracking path, and no face repeats. In a diagram of perimeter $n$ its length is at most $n/2$, since every edge of this tree path is visited, in both directions, by the boundary path.

We direct each chain towards one of its ends. A path starting at a face follows the directed chain. At the end it reaches the exterior or enters a face of a different type. In the second case it crosses neither the bottom edge nor the distinguished top edge of that face. The only such base sides are the additional $a$ in $\phi(b)=ba$ and the additional $b$ in $\phi(c)=cb$. The only possible changes are thus $a$ to $b$ and $b$ to $c$. The path then follows the new chain in its chosen direction; it reaches the exterior after traversing at most three chains, hence after $O(n)$ steps.

For a graph with boundary, let $M$ be the sum of the degrees off the boundary and $R$ the greatest distance to the boundary. Writing each value of a function that vanishes on the boundary as a sum of differences along a path of length at most $R$ gives the lower bound $1/(MR)$ for the first eigenvalue; compare Barlow, Coulhon and Grigor'yan~\cite[Lemma~2.4, (2.15), p.~619]{BarlowCoulhonGrigoryan2001}. On least-area diagrams the bounds $R\lesssim n$ and $M\lesssim n^4$ give only $cn^{-5}$ this way. We instead count the paths that traverse each dual edge. Let $K_x$ bound the number of paths meeting any one chain of type $x$, including paths starting there. A chain has at most $n$ faces. An $a$-chain receives only paths starting on that chain. Each face of a $b$-chain has just one additional side through which an $a$-chain can enter. Likewise each face of a $c$-chain has only one side through which a $b$-chain can enter. Consequently we may take
\[
       K_a=n,\qquad K_b=n(1+K_a),\qquad K_c=n(1+K_b).
\]
A path crossing a given dual edge meets the chain containing the face it leaves. There are at most two such faces and hence at most two such chains, so the edge occurs in $O(n^3)$ paths. Each path ends by crossing a side on the outer boundary, and there are at most $n$ such sides. There is one path starting from each face, so counting paths by the last side they cross gives $\Area(\Delta)\lesssim n^4$.

For a function $f$ on the faces that vanishes on the exterior, we write $f(F)$ as the sum of its successive differences along the path $\gamma_F$ from $F$ to the exterior. We apply Cauchy--Schwarz along each path, multiply by $m^*(F)\le5$, and sum over the starting faces:
\[
 \|f\|_{2,m^*}^2
 \le 5\sum_F |\gamma_F|\sum_{e\in\gamma_F}|d_e f|^2
 \le Cn\sum_e \#\{F:e\in\gamma_F\}\,|d_e f|^2
 \le C'n^4\mathcal E^*(f).
\]
Hence $\mu_1\gtrsim n^{-4}$; \eqref{eq:face-averaging} gives $\lambda_1\gtrsim n^{-4}$ if the diagram has an interior vertex. In particular both bounds hold for least-area diagrams, since these are reduced.

\smallskip
\noindent\emph{A word and its least-area filling.}
For the upper bounds we construct a least-area diagram and a vertex function with small Rayleigh quotient. Bridson's construction joins two trapezia along their last rows~\cite[Example~1.6, p.~446, and Remark~3.5, p.~452]{BridsonPolynomial2002}. We add alternating bands at the start of each trapezium. These bands have short rows; we can vary the test function between these rows while keeping its energy bounded. In the remaining bands the rows are long, but the function will be constant along the stable edges there. The initial short rows are the reason for the alternating letters in the following words:
\[
 U=(st^{-1})^m s^m,\qquad V=(ts^{-1})^m t^m,
 \qquad w_m=U^{-1}c^m U V^{-1}c^{-m}V .
\]
The word $w_m$ is cyclically reduced and has length $14m$. For each of the words $U$ and $V$ we start a trapezium with a row labelled $c^m$. If $W$ is a prefix of $U$ or $V$, respectively, then the row after $W$ reads $\phi^{e(W)}(c^m)$, where $e(W)$ is the sum of the exponents of $s$ and $t$. A positive step applies $\phi$ across one band of faces; a negative step uses the corresponding band in reverse. For every prefix $W$ we have $e(W)\ge0$. Both $U$ and $V$ have exponent sum $m$, so the two last rows have identical labels. We identify these rows left to right, matching corresponding letters. The outer arc of the first trapezium, from the left end of the common row to its right end, reads $U^{-1}c^mU$; the outer arc of the second, returning from right to left, reads $V^{-1}c^{-m}V$. The glued diagram thus has boundary word $w_m$.

We show that this diagram has least area. In the Cayley tree of $F(s,t)$ the boundary runs out and back along the geodesic to $U^{-1}$ and then out and back along the geodesic to $V^{-1}$. The two arms meet only at $1$, since the last letters of $U$ and $V$ are distinct. The tree projection identifies the two sides of each trapezium, but the sides are distinct in the full Cayley graph. At a prefix $W$ on the $U$-arm their vertices are $U^{-1}W$ and $U^{-1}c^mW$, whose difference is
\[
 (U^{-1}W)^{-1}(U^{-1}c^mW)=W^{-1}c^mW=\phi^{e(W)}(c^m)\ne1.
\]
On the other arm the same calculation uses $c^{-m}$. At the common tree vertex the two visits likewise differ by $\phi^m(c^m)$, and at each tip the path labelled $c^{\pm m}$ is simple in the free base. Distinct tree vertices cannot give the same group element. This proves that the boundary in the Cayley graph is simple; in particular no stable boundary edge in any filling can be a bridge.

Now any reduced filling has precisely two stable boundary edges over each arm edge, and no closed band. A band projects to a single tree edge, and the two boundary occurrences over that edge must be the ends of the same band. For a step between exponent sums $i$ and $j=i\pm1$, the bottom path of the band joins the prescribed boundary vertices at exponent $\min(i,j)$: the positive stable step goes from bottom to top. In the free base the label on the bottom path therefore represents $\phi^{\min(i,j)}(c^m)$ up to inversion. A band has one face per letter on the bottom path, so its number of faces is at least the freely reduced length of that word. In the constructed bands the words on the bottom paths are positive and have exactly this length. The bands partition the faces; summing these lower bounds proves that the constructed disk is least-area.

There are $2m$ alternating bands in each half, each of $m$ faces, between the rows $c^m$ and $(cb)^m$. In the next $m$ bands the exponents run from $0$ to $m$. Since $\phi^r(b)=ba^r$,
\[
 \phi^j(c)=c\prod_{r=0}^{j-1}(ba^r),\qquad
 \ell_j:=|\phi^j(c)|=1+\frac{j(j+1)}2 .
\]
The band from exponent $j$ to exponent $j+1$ has $m\ell_j$ faces for $0\le j<m$. The area of the glued diagram is consequently
\[
       4m^2+2m\sum_{j=0}^{m-1}\ell_j
         =\frac{m^4+17m^2}{3}.
\]
Together with the bound for all reduced diagrams this gives $\delta_{\mathcal P}(n)\asymp n^4$, with $m=\lfloor n/14\rfloor$ for the lower bound.

\smallskip
\noindent\emph{A test function on the vertices of the glued diagram.}
Each row is labelled by a power $\phi^j(c)^m$. Its $m$ copies of $\phi^j(c)$ divide it into successive intervals, one for each column. In row $k$, counted from $k=0$ at $c^m$ up to $k=3m$ at the common last row, we use the factors
\[
 a_i=\min\{1,3i/m,3(m-i)/m\}\quad(0\le i\le m),
 \qquad b_k=\min\{k/(2m),1\}.
\]
In column $i$, $0\le i<m$, the vertices from the left endpoint up to but not including the right endpoint have value $a_i b_k$; the final endpoint of the row has value $a_m b_k$. In other words, a shared endpoint belongs to the column on its right. We use the same rule on both halves, so the values agree along the common row. The factors $b_0=0$ and $a_0=a_m=0$ ensure that the function vanishes on the initial rows and on the sides labelled $U$ or $V$.

The endpoints of a stable edge belong to corresponding columns in the two rows, except for the edge joining the final endpoints of the rows, whose values are given separately above. Indeed, the band matches corresponding copies of $\phi^j(c)$ and $\phi^{j+1}(c)$, with left endpoints matched to left endpoints and right endpoints to right endpoints. The same holds when the band is read in reverse. The function can change along a stable edge only when $b_k$ changes, in the first $2m$ bands of either half. Each of these bands has $m+1$ stable edges. Along a base row, by contrast, $b_k$ is constant and only the last edge of each column can change $a_i$. There are $O(m)$ rows and $m$ columns, even though some rows have $\asymp m^3$ edges. In total $O(m^2)$ edges have nonzero difference, each of size at most $3/m$. The vertex energy is $O(1)$.

For all sufficiently large $m$, we restrict to columns $m/3\le i\le 2m/3$ and rows with exponent $m/2\le j<m$. There both factors are one. Each such column contributes $\ell_j\gtrsim m^2$ distinct vertices in that row, since we count the left endpoint but not the right endpoint of each column. In one half alone there are at least $cm$ such columns and at least $cm$ such rows. Hence the squared norm with respect to the degree measure is at least $cm^4$. On this least-area disk the quotient gives $\lambda_1\lesssim m^{-4}$, and \eqref{eq:face-averaging} gives $\mu_1\lesssim m^{-4}$. With $m=\lfloor n/14\rfloor$ these are the upper bounds for $\Lambda_{\mathcal P}$ and $\Lambda^*_{\mathcal P}$.

\smallskip
\noindent\emph{A group with cubic intrinsic diameter function.}
Taking $\alpha=4$ in Corollary~1.4 of Sapir, Birget and Rips~\cite[pp.~346, 352]{SapirBirgetRips2002} gives a group with Dehn function of order $n^4$ and intrinsic diameter function of order $n^3$. Here one first takes the least diameter among all fillings of each word and then the maximum over null words of length at most $n$. Both orders hold in any finite presentation $\mathcal Q$ of this group. For invariance of the intrinsic diameter function see also Bridson and Riley~\cite[Theorem~9.1, p.~152]{BridsonRiley2009}. On every least-area filling of perimeter at most $n$ we have $M\lesssim n^4$ for both walks. For either walk with a nonempty Dirichlet spectrum, \eqref{eq:cheeger-quartic} therefore gives a first eigenvalue at least $cn^{-6}$.

For the upper bounds, a vertex far from the boundary supplies a test function supported on a ball, but we also need a lower bound for its norm. The quartic isoperimetric inequality gives such a bound: a ball of radius comparable to $n^3$ about this vertex has degree measure at least $cn^4$, hence at least a fixed fraction of the total degree measure.

For every sufficiently large $n$ there is a word of length at most $n$ whose fillings all have intrinsic diameter at least $cn^3$. Indeed, by equivalence of the intrinsic diameter function with $n^3$, there are constants $C,c_0>0$ such that for each sufficiently large $k$ there is a word of length at most $Ck$ for which every filling has diameter at least $c_0 k^3$. Taking $k=\lfloor n/C\rfloor$ gives the claimed lower bound. We may pass to its cyclic reduction, since from a filling of the reduced word we can recover a filling of the original word by reattaching spurs and the conjugating path. This increases the diameter by at most $n$. For any filling of perimeter at most $n$, let $R$ be the greatest distance from a vertex to the outer boundary. Given any two vertices, we join each to the boundary by a shortest path. We choose an occurrence of each endpoint in the boundary circuit and take the shorter part of the circuit between these occurrences. Thus the diameter is at most $2R+n/2$. On a least-area filling $\Delta$ of the reduced word, this gives $R\ge c'n^3$.

We choose a vertex $v$ at distance $R$ from the boundary and an integer $r\asymp n^3$ strictly smaller than $R$. Let $B_k$ be the set of vertices in the ball of radius $k$ about $v$ in the graph metric of the diagram. For $k<r$, the ball $B_{k+1}$ is disjoint from the outer boundary. Each edge leaving $B_k$ contributes at its outside endpoint to $m(B_{k+1}\setminus B_k)$, so Lemma~\ref{lem:set-iso} gives
\[
 m(B_{k+1})-m(B_k)\ge |\partial_E B_k|\ge c\,m(B_k)^{1/4}.
\]
It follows by induction that $m(B_k)\ge b(k+1)^{4/3}$. The increment of $b(k+1)^{4/3}$ is at most $C b(k+1)^{1/3}$. Under the inductive hypothesis the recurrence gives $m(B_{k+1})-m(B_k)\ge c b^{1/4}(k+1)^{1/3}$, so a fixed sufficiently small $b$ suffices.

The test function $f=(1-d(v,\cdot)/r)_+$ is at least $1/2$ on $B_{\lfloor r/2\rfloor}$. Its squared norm is therefore at least $cr^{4/3}\gtrsim n^4$. Across each edge $f$ changes by at most $1/r$. Each edge on which $f$ varies has an interior endpoint, giving $\mathcal E(f)\le M/r^2\lesssim n^4/r^2$. The Rayleigh quotient of $f$ gives $\lambda_1(\Delta)\lesssim r^{-2}\lesssim n^{-6}$. By \eqref{eq:face-averaging} the face eigenvalue has the same upper bound. This proves Theorem~\ref{thm:quartic-least-area}.

\smallskip
\noindent\emph{The spectral Dehn functions.}
Fix admissible Rips presentations of $B_3$ and $Q$. Let $w$ be a null word of length at most $n$ in the generators of the chosen presentation of $B_3$. By Proposition~\ref{prop:first-face-qi}, there is a null word $v$ over \eqref{eq:quartic-presentation} of length at most $Cn$. Choose a least-area filling $E$ of $v$. If $E$ has faces, it is reduced and $\mu_1(E)^{-1}\le C'n^4$ by the path count above. If it has no faces, the reciprocal is zero by our convention. By the proposition, $w$ has a filling $\Delta$ with $\mu_1(\Delta)^{-1}\le C(1+\mu_1(E)^{-1})$. Thus $\Theta_{B_3}(n)\lesssim n^4$.

For the second group the cubic lower bound for the intrinsic diameter function holds in the chosen Rips presentation as well. We take a word of length at most $n$ for which every filling has diameter at least $cn^3$, and let $\Delta$ be any filling. By the diameter bound $2R+n/2$ proved above, there is a vertex $v$ at distance $R\gtrsim n^3$ from the outer boundary. We choose a face $F$ incident to $v$ and take a shortest dual path of $k$ steps from $F$ to the exterior. Starting at $v$, we follow the boundary of each successive face to an endpoint of the side crossed next. Each portion has length at most $L$, so we obtain a path from $v$ to the outer boundary of length at most $Lk$. Thus the distance from $F$ to the exterior in the dual graph is $\gtrsim n^3$. Write $M$ for the total face degree. If $M\le n^7$, we apply \eqref{eq:deep-cutoff} with $A=\{F\}$, whose mass is at least one. Then $\mu_1(\Delta)\lesssim (\log n)^2/n^6$. If $M>n^7$, the function equal to one on every face has quotient at most $n/M<n^{-6}$. The resulting bound holds on every filling of the chosen word, so $\Theta_Q(n)\gtrsim n^6/(\log n)^2$.

\end{proof}

The least-area face profile $\Lambda^*_{\mathcal P}$ need not agree for two presentations of a free group. The presentation $\langle a\mid\,\rangle$ has no diagrams with faces, whereas $z$ bounds a least-area monogon in $\langle a,z\mid z\rangle$. Thus the least-area face profile is infinite in one presentation and of order one in the other. For a fixed group with a superquadratic Dehn function, can a change of presentation change the order of either least-area profile at $j=1$? For the spectral Dehn function itself, what is the order of $\Theta_{B_3}$?

For the spectral Dehn function, Theorem~\ref{thm:intro-optimized} gives groups that attain the lower bound in Theorem~\ref{thm:intro-invariant}. Is there a finitely presented group $G$ for which
\[
            \frac{n^2}{(\log n)^2}=o\bigl(\Theta_G(n)\bigr)
            \qquad\text{and}\qquad
            \Theta_G(n)=o(n^2)?
\]

\subsection*{Acknowledgements}
The author thanks IIT Bombay for providing ideal working conditions.

\end{document}